\documentclass[a4paper,reqno]{amsart}
\usepackage{amssymb, graphicx, enumerate, enumitem, color}
\usepackage[usenames,dvipsnames,svgnames,table]{xcolor}
\usepackage[utf8]{inputenc}
\usepackage[T1]{fontenc}
\usepackage[foot]{amsaddr}
\usepackage{relsize}
\usepackage{comment}
\usepackage{todonotes}
\usepackage{amsfonts}
\usepackage[hmargin=3cm,vmargin=3cm]{geometry}
\usepackage[unicode]{hyperref}
\usepackage{xcolor}
\hypersetup{
    pdftitle={Erdős-Hajnal properties for powers of sparse graphs},
    pdfauthor={Marcin Briański, Piotr Micek, Michał Pilipczuk, Michał T. Seweryn},
    colorlinks,
    linkcolor={red!50!black},
    citecolor={blue!50!black},
    urlcolor={blue!80!black}
}

\usepackage{amsthm, thmtools}
\usepackage{nameref, cleveref}
\declaretheorem[name=Theorem, refname={Theorem,Theorems}, Refname={Theorem, Theorems}]{theorem}
\declaretheorem[name=Lemma, refname={Lemma, Lemmas}, Refname={Lemma, Lemmas}, sibling=theorem]{lemma}

\declaretheorem[name=Conjecture, refname={Conjecture, Conjectures}, Refname={Conjecture, Conjectures}, sibling=theorem]{conjecture}
\declaretheorem[name=Corollary, refname={Corollary, Corollaries}, Refname={Corollary, Corollaries}, sibling=theorem]{corollary}

\declaretheorem[name=Proposition, refname={Proposition, Propositions}, Refname={Proposition, Propositions}, sibling=theorem]{proposition}

\declaretheorem[name=Claim, refname={Claim, Claims}, Refname={Claim, Claims}]{claim}

\newtheorem*{theorem*}{Theorem}
\usepackage{mathtools}
\DeclarePairedDelimiter\set{\{}{\}}
\DeclarePairedDelimiter\ceil{\lceil}{\rceil}
\DeclarePairedDelimiter\floor{\lfloor}{\rfloor}

\newcommand{\norm}[1]{{\left|#1\right|}}

\newcommand{\calC}{\mathcal{C}}
\newcommand{\Cc}{\calC}
\newcommand{\calD}{\mathcal{D}}
\newcommand{\Dd}{\calD}

\newcommand{\calF}{\mathcal{F}}

\newcommand{\calX}{\mathcal{X}}

\newcommand{\NN}{\mathbb{N}}
\newcommand{\RR}{\mathbb{R}}

\newcommand{\N}{\NN}

\newcommand{\FO}{\mathsf{FO}}
\newcommand{\Oh}{\mathcal{O}}

\DeclareMathOperator\wcol{wcol}

\DeclareMathOperator\dist{dist}
\DeclareMathOperator\WReach{WReach}

\DeclareMathOperator\profile{profile}

\usepackage{wasysym}
\def\cqedsymbol{\ifmmode$\Diamond$\else{\unskip\nobreak\hfil
\penalty50\hskip1em\null\nobreak\hfil$\Diamond$
\parfillskip=0pt\finalhyphendemerits=0\endgraf}\fi}

\newcommand{\cqed}{\renewcommand{\qed}{\cqedsymbol}}

\let\le\leqslant
\let\ge\geqslant
\let\leq\leqslant
\let\geq\geqslant

\let\epsilon\varepsilon
\let\epsi\varepsilon
\let\eps\varepsilon

\let\phi\varphi

\renewcommand{\leq}{\leqslant}
\renewcommand{\geq}{\geqslant}
\renewcommand{\le}{\leqslant}
\renewcommand{\ge}{\geqslant}

 \renewenvironment{enumerate}{\begin{enumorig}[label=\textup{(\roman*)}, noitemsep, topsep=2pt plus 2pt, labelindent=.2em, leftmargin=*, widest=iii]}{\end{enumorig}}

\makeatletter
\let\old@setaddresses\@setaddresses
\def\@setaddresses{\bigskip\bgroup\parindent 0pt\let\scshape\relax\old@setaddresses\egroup}
\makeatother

\newcommand{\GrantThanks}{\thanks{
\begin{minipage}{.67\textwidth}The work of Micha\l{} Pilipczuk is supported by the project {\sc{Total}} that has received funding from the European Research Council (ERC) 
under the European Union's Horizon 2020 research and innovation programme (grant agreement No. 677651).
\end{minipage}\hfill\begin{minipage}{.25\textwidth}\includegraphics[width=\textwidth]{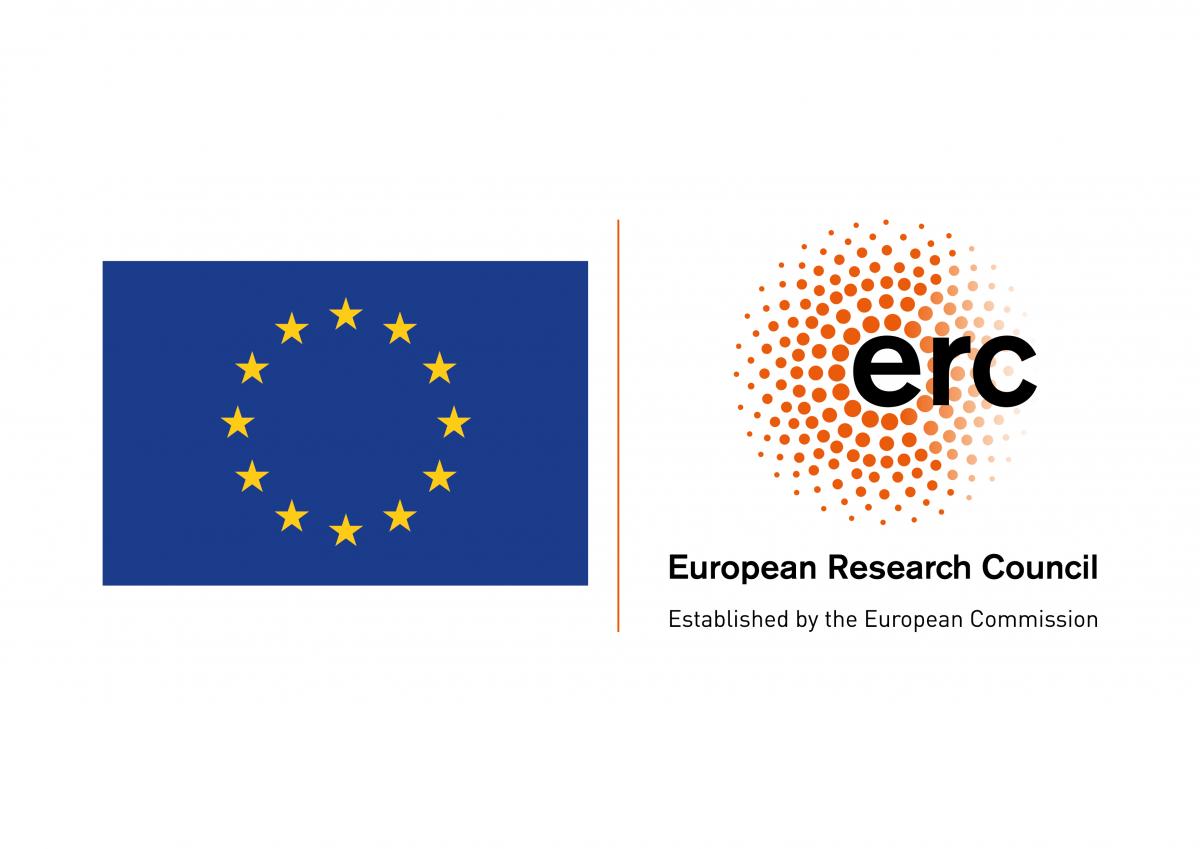}\end{minipage}\hfill}}

\begin{document}
%\title{Almost strong Erd\H{o}s-Hajnal property for powers of sparse graphs}
\title{Erd\H{o}s-Hajnal properties for powers of sparse graphs}
\GrantThanks

\author[M.~Briański]{Marcin Briański}
\address[M.~Briański,\ P.~Micek,\ M.T.~Seweryn]{Theoretical Computer Science Department, 
Faculty of Mathematics and Computer Science, Jagiellonian University, Krak\'ow, Poland}
\email{marcin.brianski@student.uj.edu.pl}

\author[P.~Micek]{Piotr Micek}
\email{piotr.micek@tcs.uj.edu.pl}

\author[Mi.~Pilipczuk]{Michał Pilipczuk}
\address[M.~Pilipczuk]{Institute of Informatics, University of Warsaw, Poland}
\email{michal.pilipczuk@mimuw.edu.pl}

\author[M.T.~Seweryn]{Michał T.\ Seweryn}
\email{michal.seweryn@tcs.uj.edu.pl}

\date{\today}

% 06A07 Combinatorics of partially ordered sets based within the same
% 05C35 Extremal problems
%\subjclass[2010]{06A07, 05C35}

%\keywords{}

\begin{abstract} 
We prove that for every nowhere dense class of graphs $\Cc$, positive integer $d$, 
and $\eps>0$, the following holds: 
in every $n$-vertex graph $G$ from $\Cc$ one can find two disjoint vertex subsets $A,B\subseteq V(G)$ such that 
\begin{itemize}
 \item  $|A|\geq (1/2-\eps)\cdot n$ and $|B|=\Omega(n^{1-\eps})$; and
 \item either $\dist(a,b)\leq d$ for all $a\in A$ and $b\in B$, or $\dist(a,b)>d$ for all $a\in A$ and $b\in B$.
\end{itemize}
We also show some stronger variants of this statement, including a generalization to the setting of First-Order interpretations of nowhere dense graph classes. 
\end{abstract}

\maketitle

\section{Introduction}

\subsection*{Sparsity and density.}
The theory of structural sparsity in graphs revolves around two notions introduced by Ne\v{s}et\v{r}il and Ossona de Mendez in~\cite{NesetrilM08a,NesetrilM08b,NesetrilM11}: {\em{bounded expansion}} and {\em{nowhere denseness}}.
We say that a class of graphs $\Cc$ is {\em{nowhere dense}} if for every $d\in \N$, there is an upper bound $t(d)\in \N$ on the sizes of cliques that can be found as depth-$d$ minors of graphs from $\Cc$. Here,
$H$ is a {\em{depth-$d$ minor}} of $G$ if $H$ can be obtained from a subgraph of $G$ by contracting mutually disjoint connected subgraphs of radius at most $d$.
More restrictively, $\Cc$ has {\em{bounded expansion}} if for every $d\in \N$ there is an upper bound $c(d)\in \N$ on the average degree of depth-$d$ minors of graphs from $\Cc$.
The concept of bounded expansion encompasses a vast majority of well-studied classes of sparse graphs.
Examples include proper minor-closed classes, classes with bounded maximum degree,
%However, there are more general examples, for instance the following classes have bounded expansion as well: 
graphs that can be drawn in the plane with a bounded number of crossings per edge~\cite{NOdMW12}, and
intersection graphs of bounded-ply families of fat objects in Euclidean spaces~\cite{Har-PeledQ17}. 

While the definitions of nowhere denseness and of bounded expansion may seem arbitrary at first glance, the last decade has witnessed a prodigious development of various structural techniques for them.
Apart from describing many interesting combinatorial properties, these techniques have strong applications in algorithm design and uncover deep links with logical aspects of sparsity.
In particular, it has been shown~\cite{DKT12,GroheKS17} that for subgraph-closed classes of graphs, nowhere denseness exactly%
\footnote{
Assuming that $\FO$ model-checking cannot be solved in FPT time on all graphs
(which is implied by
the widely believed conjecture that $\mathsf{FPT} \neq \mathsf{W}[1]$,
or even $\mathsf{AW}[\star]\neq \mathsf{FPT}$).
}
delimits classes where every property expressible in the First-Order logic $\FO$
can be tested in almost linear time.
This shows that nowhere denseness corresponds to an important dividing line in descriptive complexity theory.
We refer to the monograph of Ne\v{s}et\v{r}il and Ossona de Mendez~\cite{sparsity} and to more recent lecture notes~\cite{notes} for a broader introduction to the subject.

The techniques of structural sparsity can be applied not only to describe structure in sparse graphs, but also to analyze dense graphs that can be derived from sparse ones using simple combinatorial transformations.
Take, for instance, the concept of the power of a graph: for a graph $G$ and an integer $d\in \N$, the {\em{$d$th power}} $G^d$ is the graph on the same vertex set as $G$, where two vertices $u$ and $v$ are considered adjacent 
if and only if they are at distance at most $d$ in $G$. Taking a power can turn a sparse graph into a dense one: e.g.\ the second power of a star is a clique.
Thus, if $\Cc$ is a class of sparse graphs, then the class $\Cc^d$ of the $d$th powers of graphs in $\Cc$ may contain dense graphs, but the intuition is that graphs from $\Cc^d$ should admit strong structural properties due to admitting
sparse pre-images. This work fits into the broader direction of trying to describe and quantify these properties.

Kwon et al.~\cite{KPS20} proved a useful structural theorem for powers of classes of bounded expansion, namely they admit {{low shrubdepth colorings}}\footnote{In~\cite{KPS20}, Kwon et al. discuss 
a weaker notion of {\em{low rankwidth colorings}} and only mention that the colorings they obtain for powers of classes of bounded expansion are actually low shrubdepth colorings. This has been later clarified and generalized
by Gajarsk\'y et al. in~\cite{GajarskyKNMPST18}.}.
This mirrors the existence of so-called {\em{low treedepth colorings}} in classes of bounded expansion, see~\cite{NesetrilM08a}.

Shrubdepth was introduced by Ganian et al.~\cite{GanianHNOM19} and, informally, a graph has small shrubdepth if its adjacency relation can be concisely described in a tree of bounded depth, called the {connection model}.
A \emph{connection model} for a graph \(G\) with a label set \(\Lambda\) is a tuple \((\lambda, T, \set{Z_x \colon x \in V(T) \setminus V(G)})\) consisting of:
\begin{enumerate}
\item a labelling \(\lambda\) of vertices of \(G\) with labels from \(\Lambda\),
\item a rooted tree \(T\) whose leaves are the vertices of \(G\), and
\item for every non-leaf node $x$ of $T$, a symmetric binary relation $Z_x\subseteq \Lambda^2$,
\end{enumerate}
such that the following condition holds. For every pair of vertices \(u, v\) in $G$, we have
\[uv \in E(G) \Leftrightarrow (\lambda(u), \lambda(v)) \in Z_x,\quad
\textrm{where \(x\) is the lowest common ancestor of \(u\) and \(v\) in \(T\).}\]
The \emph{shrubdepth}\footnote{The original definition proposed by Ganian et al. in~\cite{GanianHNOM19} is slightly different and applies only to classes of graphs, and not single graphs. The variant defined here leads to the same notion of having low shrubdepth colorings.} of $G$ is the lowest number $s$ such that $G$ has a connection model of depth at most~$s$ over a label set of size $s$.
A class $\Cc$ admits \emph{low shrubdepth colorings} if for every $p\in \N$ there exist numbers $f(p)$ and $s(p)$ such that every graph $H$ in $\Cc$ has a vertex coloring with at most $f(p)$ colors in which 
every subset of $p$ colors induces a subgraph of shrubdepth at most $s(p)$.
%Here, {\em{shrubdepth}} is a graph parameter introduced by Ganian et al.~\cite{GanianHNOM19} that is the dense analogue of treedepth. 

The result of Kwon et al.~\cite{KPS20} was later generalized by Gajarsk\'y et al.~\cite{GajarskyKNMPST18} to arbitrary one-dimensional $\FO$ interpretations of classes of bounded expansion, 
where the binary predicate ``$\dist(u,v)\leq d$'' that is used in the definition of the $d$th power can be replaced with any property of a pair of vertices that can be expressed in $\FO$.
This encompasses a broad range of graph transformations definable in logic.

The existence of low shrubdepth colorings with a constant number colors entails a number of strong structural properties.
For instance, as observed in~\cite{KPS20}, it implies that the class in question is $\chi$-bounded.
Mimicking the situation for low treedepth colorings, one should expect an analogue of these results for powers, or even $\FO$ interpretations, of nowhere dense classes. 
Therefore, the following natural conjecture, posed in
\cite{KPS20}, is the goal
which sets the motivation for our work.

\begin{conjecture}\label{conj:lsc}
Let $\Cc$ be a nowhere dense class of graphs, $d$ be a positive integer, $\eps$ be a positive real. 
Then there is a function $f\colon \N\to \N$ such that for every natural $p$ and 
every $n$-vertex graph $G$ in $\Cc^d$, there is a coloring of $G$ using $\Oh(n^{\eps})$ colors in which every subset of $p$ colors induces a subgraph of $G$ of shrubdepth at most $f(p)$.
\end{conjecture}

Unfortunately, the arguments provided in~\cite{KPS20,GajarskyKNMPST18} do not generalize to the nowhere dense setting and Conjecture~\ref{conj:lsc} remains open.
While this statement would expose interesting and potentially quite useful structure in powers of nowhere dense classes, 
one can also investigate those hypothetical corollaries and try to prove them using different means, in hope of gathering more insight.
Our work contributes to this line of research.
%This general goal sets the motivation of our work.

%Therefore, we would be interested in extending the results of Kwon et al.~\cite{KPS20} and of Gajarsk\'y et al.~\cite{GajarskyKNMPST18} to more general settings, for instance to powers, or even $\FO$ interpretations, of
%nowhere dense classes.

\subsection*{Our results.} We study properties of powers of nowhere dense classes related to the Erd\H{o}s-Hajnal conjecture. 
Recall that a graph class $\Cc$ has the {\em{Erd\H{o}s-Hajnal property}} if there exists $\delta>0$ such that every graph $G$ in $\Cc$ on $n$ vertices contains either a clique or an independent set of size at least~$n^\delta$.
The {\em{Erd\H{o}s-Hajnal conjecture}}~\cite{EH89} hypothesizes that for every fixed graph $H$, the class of graphs that exclude $H$ as an induced subgraph has the Erd\H{o}s-Hajnal property.
The conjecture is open for as simple graphs as $H=C_5$ and $H=P_5$, see the survey of Chudnovsky~\cite{Chudnovsky14} for an overview.

It is known that for every nowhere dense class $\Cc$ and integer $d\in \N$, the class $\Cc^d$ excludes some half-graphs as a semi-induced subgraph~\cite{AdlerA14}, and hence has the Erd\H{o}s-Hajnal property~\cite{MalliarisSh14} (see also~\cite[Chapter 5, Theorem 2.8]{notes} for a compact explanation).
However, when $\Cc$ has bounded expansion, it is a simple consequence of the existence of low shrubdepth colorings that the class $\Cc^d$ actually enjoys the strong Erd\H{o}s-Hajnal property defined as follows.
A class of graphs admits the {\em{strong Erd\H{o}s-Hajnal property}} if there exists a constant $\delta>0$ such that in every $n$-vertex graph $G$ with $n \ge 2$ from the class there exist disjoint vertex subsets $A,B$, each of size
at least $\delta n$, such that either every vertex of $A$ is adjacent to every vertex of $B$ (i.e.\ $A$ is {\em{complete}} to $B$),
or there is no edge with one endpoint in $A$ and the other in $B$ (i.e.\ $A$ is {\em{anticomplete}} to $B$).
It is known that for hereditary classes, the strong Erd\H{o}s-Hajnal property implies the (standard) Erd\H{o}s-Hajnal property~\cite{APPRS05}, but the converse implication is not true e.g.\ for triangle-free graphs.

Let us note that while the connection between low shrubdepth colorings and the strong Erd\H{o}s-Hajnal property was observed during the work on~\cite{KPS20,GajarskyKNMPST18}, it has not been included in any published manuscript. Hence, for the sake of completeness we clarify it in Appendix~\ref{app:shrubs}.

Conjecture~\ref{conj:lsc}, if true, would imply a weaker variant of the strong Erd\H{o}s-Hajnal property for powers of nowhere dense classes, where the lower bound of $\delta n$ on the cardinalities of $A$
and $B$ would be replaced by $n^{1-\eps}$, for any $\eps>0$ fixed in advance and sufficiently large $n$. The main result of this work is a direct proof of this statement; in fact, we show its significant generalization.

\begin{theorem}\label{thm:main}
Let $\calC$ be a nowhere dense graph class, $d$ be a positive integer, $\epsi$ be a positive~real.
Then there is an integer $k_0$ such that for every graph $G$ in $\Cc$ and vertex subsets $A$, $B$ in $G$ satisfying $|A \cup B|\geq k_0$,
there exist disjoint sets $A'$ and $B'$ with $A' \subseteq A$, $B'\subseteq B$ such that
$$|A'| \ge (1-\eps)\cdot |A|,\qquad\qquad |B'| \ge \frac{|B|}{|A \cup B|^{\eps}},$$
and in $G^d$, every vertex from $A'$ is adjacent either to all vertices from $B'$, or to no vertex from $B'$.
\end{theorem}

If sets \(A'\) and \(B'\) satisfy Theorem~\ref{thm:main}, then we can find a subset
\(A''\) consisting of at least half of vertices of \(A'\) which is either complete or
anticomplete to \(B'\) in \(G^d\).
Hence, applying the theorem with \(A = B = V(G)\) implies that
in every sufficiently large \(n\)-vertex graph \(G\) from a fixed nowhere dense
class there are two disjoint sets which are either complete or anticomplete
to each other in \(G^d\) where the lower bounds on the sizes of the sets
are \((1/2 - \eps)n\) and \(n^{1-\eps}\), respectively.
Moreover, we may choose \(A'\) and \(B'\) from any prescribed vertex subsets
\(A\) and \(B\) and the lower bounds will be scaled appropriately.

%Observe that in the statement of Theorem~\ref{thm:main}, we may choose $A'$ and $B'$ from any prescribed vertex subsets $A$ and $B$.
%Moreover, we can require that $A'$ is as large as almost half of $A$; only for the cardinality of $B$ the lower bound is slightly sublinear.

The main technical step towards the proof of Theorem~\ref{thm:main} is the following lemma, which may be of independent interest.
Here, we say that a vertex subset $S$ {\em{$d$-separates}} vertex subsets $A$ and $B$ in a graph $G$ if every path of length at most $d$ with one endpoint in $A$ and second in $B$ intersects~$S$.

\begin{lemma}\label{lem:new-abs}
  Let $\calC$ be a nowhere dense graph class, $d$ be a positive integer, $\eps$ be a positive~real.
  Then there is an integer $k_0$ such that for every graph $G$ in $\Cc$ and vertex subsets $A$, $B$ in $G$ satisfying $|A \cup B|\geq k_0$,
  there exist a vertex subset $S$ and disjoint
  vertex sets $A'$ and $B'$
  with $A'\subseteq A$, $B'\subseteq B$ such that
  $$|A'|\geq (1-\eps)\cdot |A|,\qquad |B'|\geq \frac{|B|}{|A \cup B|^\eps},\qquad |S|\leq |A \cup B|^\eps,$$ 
  and $S$ $d$-separates $A'$ and $B'$.
\end{lemma}

Theorem~\ref{thm:main} follows from combining Lemma~\ref{lem:new-abs} with known bounds on the neighborhood complexity in nowhere dense classes~\cite{EickmeyerGKKPRS17}.
However, Lemma~\ref{lem:new-abs} can be also used together with the newer bounds on the number of $\FO$ types in sparse graphs~\cite{PilipczukST18a} 
to give the following generalization of Theorem~\ref{thm:main} to $\FO$ interpretations
of nowhere dense classes (see Section~\ref{sec:interpretations} for relevant definitions).
\begin{theorem}\label{thm:interpretations}
Let $\calC$ be a nowhere dense graph class, $d$ be a positive integer, $\epsi$ be a positive~real, $\varphi(x,y)$ be an $\FO$ formula with two free variables.
Then there is an integer $k_0$ such that for every graph $G$ in $\Cc$ and vertex subsets $A$, $B$ in $G$ satisfying $|A \cup B|\geq k_0$,
there exist disjoint sets $A'$ and $B'$ with $A' \subseteq A$, $B'\subseteq B$ such that
$$|A'| \ge (1-\eps)\cdot |A|,\qquad\qquad |B'| \ge \frac{|B|}{|A \cup B|^{\eps}},$$
and for every $a \in A'$, either we have $G\models \varphi(a,b)$ for all $b\in B'$,
or $G\models \neg \varphi(a,b)$ for all $b\in B'$.
\end{theorem}

\begin{comment}
\begin{theorem*}
Let $\Cc$ be a nowhere dense class of graphs, $\varphi(x,y)$ be an $\FO$ formula with two free variables, and $\eps$ be a positive real. 
There exists an integer $k_0$ such that for every graph $G$ in $\Cc$ and vertex subset $W$ in $G$ satisfying $|W|\geq k_0$, there exist disjoint subsets $A$, $B$ with $A,B\subseteq W$ such that:
\[
|A|\geq \left(1/2-\eps\right) |W|,\qquad |B|\geq |W|^{1-\eps},
\]
and either $G\models \varphi(a,b)$ for all $a\in A$ and $b\in B$, 
or $G\models \neg \varphi(a,b)$ for all $a\in A$ and $b\in B$.

\end{theorem*}
\end{comment}

Finally, we show that unlike in the bounded expansion case, the nowhere dense classes of
graphs do not have the strong Erd\H{o}s-Hajnal property.
Therefore, in Theorem~\ref{thm:main} we cannot require that both of the subsets
\(A'\) and \(B'\) contain a fixed positive fraction of sets \(A\) and \(B\) respectively.

\begin{proposition}\label{prop:no-strong-EH}
 There exists a nowhere dense class of graphs that does not enjoy the strong Erd\H{o}s-Hajnal property.
\end{proposition}

This proposition was already shown in~\cite[Theorem~50]{JNOdMS20}, but we propose an
arguably simpler explanation of this fact.

\subsection*{Other related work.}
Very recently, Ne\v{s}et\v{r}il et al.~\cite{NOdMPZ19} proved that powers of sparse graph classes admit good clustering properties in the following sense:
if $\Cc$ is a nowhere dense class and $d\in \N$ is fixed, then every graph from $\Cc^d$ can be vertex-partitioned into cliques so that contracting every clique to a single vertex yields a graph from another
nowhere dense class $\Dd$, which depends only on $\Cc$ and $d$. The same holds also when the notion of nowhere denseness is replaced with having bounded expansion.
From this result it is possible to derive a weaker formulation of Theorem~\ref{thm:main}, where both sides are allowed to have slightly sublinear sizes, see Appendix~\ref{app:shrubs} for a discussion.

Malliaris and Shelah~\cite{MalliarisSh14} gave improved bounds for the Regularity Lemma in classes of graphs that exclude a {\em{semi-induced half-graph}} of a constant length, also known as {\em{combinatorially stable}} classes.
It is known that powers of nowhere dense classes, and even their images under $\FO$ interpretations, have this property~\cite{AdlerA14}, hence the results of Malliaris and Shelah apply to our setting.
From the results provided in~\cite{MalliarisSh14} it is possible to derive statements of flavor similar to the strong Erd\H{o}s-Hajnal property, 
however the obtained bounds would be significantly weaker than the ones provided by Theorem~\ref{thm:main}.

The ``almost'' strong Erd\H{o}s-Hajnal property of similar flavor as the one provided by Theorem~\ref{thm:main} was proved by Fox et al.~\cite{FPT10} for the class of intersection graphs of $x$-monotone curves in the plane.
Namely, if $\Cc$ is this class, then
there is a positive constant $\delta$ such that in every $n$-vertex graph $G$ in $\Cc$ with $n \geq 2$, one can find 
either find disjoint vertex subsets $A$ and $B$ such that $|A|,|B|\geq \delta n/\log n$ and $A$ is complete to $B$,
or disjoint vertex subsets $A$ and $B$ such that $|A|,|B|\geq \delta n$ and $A$ is anticomplete to $B$.
In general, the strong Erd\H{o}s-Hajnal property appears often in the setting of geometric intersection graphs, for instance
the class of intersection graphs of convex subsets of the plane enjoys it~\cite{FPT10}. See~\cite{FP08} for a survey on this line of research.

Finally, in a very recent series of papers, Scott, Seymour and Spirkl study how large {\em{pure pairs}} --- disjoint vertex subsets $A$ and $B$ such that $A$ is either complete or anticomplete to $B$ --- can be found in graphs drawn from various graph classes defined by forbidding induced subgraphs. In particular, in one of the works~\cite{ScottSS19} they prove a statement of flavor very similar to ours: they extract a pure pair $(A,B)$ such that $|A|=\Omega(n)$ and $|B|=\Omega(n^{1-\epsi})$ for any $\epsi>0$ fixed beforehand. The setting is, however, very different: they assume excluding a ``long'' subdivision of a fixed graph~$H$ as an induced subgraph, which imposes structure unlike the one found in powers of sparse~graphs.

\subsection*{Organization.} 
In Section~\ref{sec:preliminaries} we list the preliminary definitions. 
In particular, we define the key sparsity measure for our work, the weak coloring numbers. 
In Section~\ref{sec:lemmas} we prove our main technical contribution, i.e.\ Lemma~\ref{lem:new-abs}. 
In Section~\ref{sec:proofs-of-the-theorem} we wrap up the proof of Theorem~\ref{thm:main} 
and in Section~\ref{sec:interpretations} we show how to lift it to FO interpretations of nowhere dense classes.
Section~\ref{sec:example} describes a nowhere dense class without the strong Erd\H{o}s-Hajnal property.
In Appendix~\ref{app:shrubs} we include a proof that a class with low shrubdepth colorings admits the strong Erd\H{o}s-Hajnal property.

\section{Preliminaries}
\label{sec:preliminaries}

All graphs in this paper are finite, simple, and undirected.
For a graph $G$, we denote by $V(G)$ and $E(G)$ the vertex set and the edge set of $G$, respectively.
%For a subset $X\subset V(G)$ we denote by $G[X]$ the subgraph of $G$ induced by vertices in $X$.

A \emph{walk} (of \emph{length} \(k\)) in a graph \(G\) is a nonempty alternating sequence
\(v_0 e_0 v_1 e_1 \dots e_{k-1} v_k\) of vertices and edges in \(G\) such that \(e_i = v_i v_{i+1}\)
for all \(i \in \{0, \dots, k - 1\}\).
Thus, a path is a graph in which there is a walk containing each vertex and each edge exactly once.
A walk whose first vertex is \(a\) and the last vertex is \(b\) is called an \emph{\(a\)--\(b\) walk}. 
The \emph{distance} \(\dist_G(a, b)\) between two vertices \(a\) and \(b\) in $G$ is the length of a shortest \(a\)--\(b\) walk, or \(+\infty\) if there is no \(a\)--\(b\) walk in \(G\).
% (i.e.\ the vertices lie in distinct components of the graph).
For a set \(X\) of vertices in a graph \(G\), the set of all vertices at distance at most $1$ from any vertex from $X$ is denoted by $N_G[X]$, and 
the set of all vertices at distance at most \(r\) from any vertex from \(X\) is denoted by \(N_G^r[X]\).
When \(X\) is a singleton \(\{v\}\), we write \(N_G[v]\) and \(N_G^r[v]\) as shorthands for \(N_G[\{v\}]\) and \(N_G^r[\{v\}]\), respectively.
The \emph{radius} of a connected graph $G$ is the least integer $r$ for which there is a vertex $v\in V(G)$ such that $N^r_G[v]=V(G)$.
The \(d\)-th \emph{power} of a graph \(G\), denoted \(G^d\), is a graph on the vertices of \(G\) in which two distinct vertices are made adjacent if and only if the distance between them in \(G\) is at most \(d\).

Given a partition $\calX$ of the vertices of a graph $G$ into nonempty parts inducing connected subgraphs,
we denote by $G/\calX$ the graph with vertex set $\calX$ in which two distinct parts $X, Y \in \calX$ are adjacent if in $G$ there exists an edge with one endpoint in $X$ and second in $Y$.
A graph $H$ is a \emph{minor} of $G$ if $H$ is isomorphic to a subgraph of $G/\calX$ for some partition $\calX$ of $V(G)$.

We now recall the definitions of classes with bounded expansion and of nowhere dense classes.
A graph $H$ is a \emph{depth-$r$ minor} (also known as an \emph{$r$-shallow minor})  of a graph $G$
if $H$ is isomorphic to a subgraph of $G/\calX$ for some partition $\calX$ of $V(G)$ into
nonempty parts inducing subgraphs of radius at most $r$.
The \emph{greatest reduced average density} (\emph{grad}) of \emph{depth $r$} of a nonempty graph $G$, denoted by $\nabla_r(G)$, is defined as
\[\nabla_r(G)=\max\left\{\,\frac{\norm{E(H)}}{\norm{V(H)}} \ \colon\ \textrm{$H$ is a nonempty depth-$r$ minor of $G$}\,\right\}.\]
We say that a class of graphs $\calC$ has \emph{bounded expansion} if there exists a function $f \colon \NN \to \RR$ such that
$\nabla_r(G)\leq f(r)$ for all $r\geq 0$ and $G$ in $\calC$.
More generally, a class of graphs $\calC$ is \emph{nowhere dense} if for each integer $r\geq 0$ there exists a graph
which is {\em not} a depth-$r$ minor of any graph $G$ from $\calC$.

Much of the interest that these concepts have gathered in the recent years can be attributed to the realization that 
they admit multiple seemingly different characterizations.
In this work, we mostly use the characterization through {\em{weak coloring numbers}},
and, at the very end, we use some bounds for the {\em{neighborhood complexity}}.
We now proceed with the definitions of the weak coloring numbers,
while the neighborhood complexity is introduced and applied in Section~\ref{sec:proofs-of-the-theorem}.

Let $G$ be a graph and let $\sigma$ be a vertex ordering of $G$.
For a nonnegative integer $r$ and two vertices $u$ and $v$ of $G$, we say that
$u$ is \emph{weakly $r$-reachable} from $v$ in $\sigma$, if there exists
a $u$--$v$ walk of length at most $r$ 
such that for every vertex $w$ on the walk, $w$ is not smaller than $u$ in $\sigma$.
The set of vertices that are weakly $r$-reachable from a vertex $v$ in $\sigma$ is denoted by $\WReach_r[G, \sigma, v]$.
%In particular, $v$ is always in $\WReach_r[G, \sigma, v]$.
We define
\begin{align*}
\wcol_r(G, \sigma) &= \max_{v \in V(G)}\ |\WReach_r[G, \sigma, v]|,\\
\wcol_r(G) &= \min_{\sigma}\ \wcol_r(G, \sigma),
\end{align*}
where $\sigma$ ranges over the set of all vertex orderings of $G$.
We call $\wcol_r(G)$ the $r$-\emph{th weak coloring number} of $G$.

By the results of Zhu~\cite{Zhu09} and of 
Ne{\v{s}}et{\v{r}}il and Ossona de Mendez~\cite{NesetrilM11},
a class of graphs $\calC$ is nowhere dense if and only if
for every nonnegative integer $r$ and every $\epsi>0$,
there exists an integer $n_0$ such that for every $n\geq n_0$ and every $n$-vertex graph $G$ in $\calC$,
we have $\wcol_r(G) \le n^{\epsi}$.
As this holds for all \(\epsi\), this means that $\wcol_r(G) = o(n^{\epsi})$.

\section{Proof of Lemma~\ref{lem:new-abs}}
\label{sec:lemmas}

In order to prove Lemma~\ref{lem:new-abs}, we first prove a weaker version,
and then we show how to get the stronger result with a localization argument that
restricts attention to a subgraph of small size.
In the weaker version of the lemma, the \(|A \cup B|^{\epsi}\) terms are
replaced by \(|V(G)|^\epsi\) and we don't require the sets \(A'\) and \(B'\)
to be disjoint.

Recall that a vertex subset \(S\) \(d\)-separates two vertex subsets \(A\)
and \(B\) in a graph \(G\) if every path of length at most \(d\) with
endpoints in \(A\) and \(B\) intersects \(S\).
This is equivalent to saying that in the graph \(G - S\) there is no
\(a\)--\(b\) walk of length at most \(d\) with
\(a \in A \setminus S\) and \(b \in B \setminus S\).

\begin{lemma}\label{lem:weak-abs}
  Let $\calC$ be a nowhere dense class of graphs, $d$ be a positive integer, and $\eps$ be a positive~real.
  Then there is a constant $n_0\in \N$ such that for every \(n\)-vertex graph $G$ in $\Cc$ with \(n \ge n_0\) and vertex subsets $A,B\subseteq V(G)$,
  there exist a vertex subset $S\subseteq V(G)$ and sets $A'\subseteq A$ and $B'\subseteq B$ such that
  $$|A'|\geq (1-\eps)\cdot |A|,\qquad |B'|\geq \frac{|B|}{n^\eps},\qquad |S|\leq n^\eps,$$ 
  and $S$ $d$-separates $A'$ and $B'$.
\end{lemma}
\begin{proof}
  Let \(k = \ceil{1/\epsi}\).
  Since \(\wcol_{4kd}(G) = o\left(|V(G)|^{\epsi}\right)\) for graphs \(G\) in \(\calC\), we can fix an integer \(n_1\) such that \(\wcol_{4kd}(G) \le (\epsi/2) n^{\epsi}\)
  for every \(n\)-vertex graph \(G\) in \(\calC\) with \(n \ge n_1\).
  Let \[n_0 = \max \left\{n_1, \ceil*{{(2/\epsi)}^{1/\epsi}}\right\}.\]
  Let \(G\) be an \(n\)-vertex graph in \(\calC\) with \(n \ge n_0\), and let \(A\) and
  \(B\) be two vertex subsets in \(G\).
  If \(A = \emptyset\), then the lemma statement is satisfied by setting \(A' = S = \emptyset\) and \(B' = B\).
  Hence we assume that \(A \neq \emptyset\).
  As \(n \ge n_0 \ge n_1\), we have \(\wcol_{4kd}(G) \le (\epsi/2) n^{\epsi}\).

\begin{claim}\label{clm:hub-vertices}
  There is a subset \(S\) of \(V(G)\) such that \(|S| \le n^{\epsi}\) and
  \(|N^{2kd}_{G-S}[v] \cap A| \le (\epsi/2)|A|\) for every \(v \in V(G - S)\). 
\end{claim}
\begin{proof}
  Let \(\sigma\) be a vertex ordering of \(G\) such that
  \(\wcol_{4kd}(G, \sigma) = \wcol_{4kd}(G) \le (\epsi/2) n^{\epsi}\).
  For a vertex \(s\in V(G)\), let \(\WReach^{-1}_{4kd}[G,\sigma,s]\) denote the
  set of all vertices \(u \in V(G)\) such that \(s \in \WReach_{4kd}[G,\sigma,u]\).
  Let \[S=\left\{\,s\in V(G) : |\WReach^{-1}_{4kd}[G,\sigma, s] \cap A| > (\epsi/2)|A|\,\right\}.\]
  Let us estimate the number \(p\) of pairs \((s, a) \in S \times A\) such that
  \(s\in \WReach_{4kd}[G,\sigma,a]\).
  On the one hand, each \(s \in S\) is weakly \((4kd)\)-reachable from at least \((\epsi/2)|A|\)
  vertices \(a \in A\), so \(p \ge |S| \cdot (\epsi/2)|A|\).
  On the other hand, for each \(a \in A\) there are at most
  \(\wcol_{4kd}(G)\) vertices \(s \in S\) which are weakly \((4kd)\)-reachable from \(a\), so \(p \le |A| \cdot \wcol_{4kd}(G) \le |A| \cdot (\epsi/2) n^{\epsi}\).
  Combining these two inequalities yields
  \(|S| \cdot (\epsi/2)|A| \le |A| \cdot (\epsi/2) n^{\epsi}\). Hence \(|S| \le n^{\epsi}\),
  as required.

  Consider any \(v \in V(G - S)\) and let \(\ell\) be the \(\sigma\)-minimal vertex of
  \(N^{2kd}_{G-S}[v]\).
  For any vertex \(u \in N^{2kd}_{G-S}[v]\), consider the \(\ell\)--\(u\) walk
  \(W\) obtained as the concatenation of an \(\ell\)--\(v\) walk of length at
  most \(2kd\) and an \(v\)--\(u\) walk of length at most \(2kd\); both these walks are chosen in $G-S$.
  The walk \(W\) has all its vertices in \(N^{2kd}_{G-S}[v]\),
  so \(\ell\) is the \(\sigma\)-minimal vertex of the walk \(W\).
  Since the length of \(W\) is at most \(4kd\), the walk \(W\) witnesses
  that \(\ell \in \WReach_{4kd}[G, \sigma, u]\).
  As the vertex \(u\) was chosen arbitrarily, we conclude that
  \(N^{2kd}_{G-S}[v] \subseteq \WReach^{-1}_{4kd}[G, \sigma, \ell]\).
  Since \(\ell \not \in S\), this implies that
  \[|N^{2kd}_{G-S}[v] \cap A| \le |\WReach^{-1}_{4kd}[G, \sigma, \ell] \cap A| \le (\epsi/2)|A|.\]
  This completes the proof of the claim.
\cqed\end{proof}
  
  Let us fix a set \(S\) as in Claim~\ref{clm:hub-vertices}. 
  Let \[H = {(G - S)}^d,\qquad A_0 = A \setminus S,\qquad \textrm{and}\qquad B_0 = B \setminus S.\]
  Then for every \(v \in V(H)\), we have 
  \begin{equation}
  |N_H^{2k}[v] \cap A_0| = |N_{G-S}^{2kd}[v] \cap A| \le ({\epsi}/{2})|A|.
  \label{eq:small-balls}
  \end{equation}
  Note that $S$ $d$-separates two subsets \(A'\) and \(B'\) of \(V(G)\) in $G$
  if and only if $A'\setminus S$ is anticomplete to $B'\setminus S$ in $H$
  (i.e.\ the two sets are disjoint and there is no edge between them).
%  then every \(A'\)--\(B'\) path
%  of length at most \(d\) in \(G\) intersects \(S\).
  Hence, to complete the proof it suffices to show that there are subsets \(A'\) and \(B'\) of \(A\) and \(B\), respectively,
  such that \(|A'| \ge (1-\epsi)|A|\), \(|B'| \ge |B|/n^{\epsi}\), and 
  $A'\setminus S$ is anticomplete to $B'\setminus S$ in $H$.

  Let 
  $$\alpha = \frac{|A|}{|B|}\cdot ({\epsi}/{2}) n^{\epsi}\qquad\textrm{and}\qquad \beta = \frac{|B|}{|A|}\ .$$ 
  For two pairs \((X, Y)\) and \((A_1, B_1)\) with \(X \subseteq A_1 \subseteq A_0\) and \(Y \subseteq B_1 \subseteq B_0\), we say that the pair \((X, Y)\) is \emph{good in} \((A_1, B_1)\) if
  \begin{enumorig}[label={(G\arabic*)}]
      \item\label{itm:good-xy} $X$ is anticomplete to $Y$ in $H$,
      \item\label{itm:good-x} \(|N_H[X] \cap B_1| \le \beta |X|\), and
      \item\label{itm:good-y} \(|N_H[Y] \cap A_1| \le \alpha |Y|\).
  \end{enumorig}
  \begin{claim}\label{clm:good-union}
    If \((X_0, Y_0)\) is a good pair in \((A_0, B_0)\) and \((X_1, Y_1)\) is a good pair in \((A_1, B_1)\), where \(A_1 = A_0 \setminus (X_0 \cup N_H[Y_0])\) and \(B_1 = B_0 \setminus (Y_0 \cup N_H[X_0])\), then \((X_0 \cup X_1, Y_0 \cup Y_1)\) is a good pair in \((A_0, B_0)\).
    See~Figure~\ref{figs-extending-a-good-pair}.
  \end{claim}
  \begin{figure}[!h]
  \includegraphics{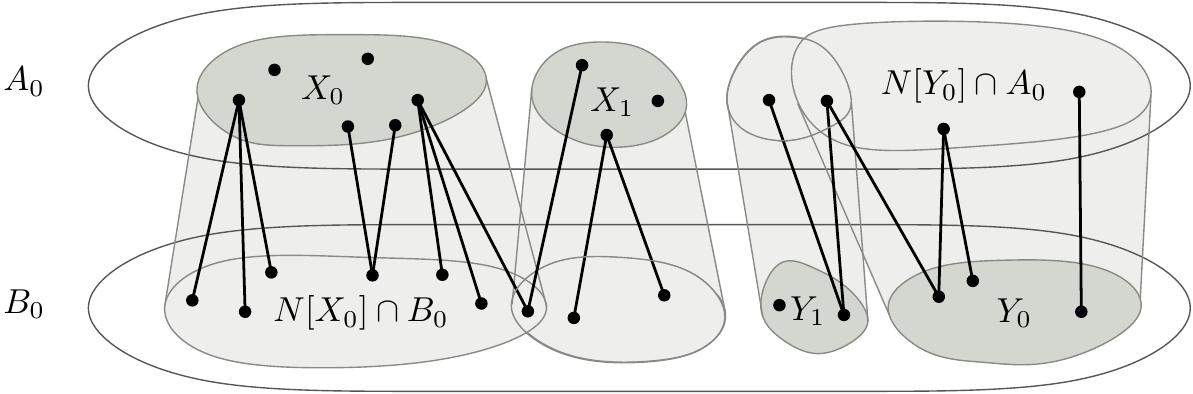}
  \caption{Two good pairs: $(X_0,Y_0)$ a good pair in $(A_0,B_0)$, and $(X_1,Y_1)$ a good pair in $(A_0\setminus (X_0\cup N[Y_0]), B_0 \setminus(Y_0\cup N[X_0]))$.
  \label{figs-extending-a-good-pair}}
  \end{figure}
  \begin{proof}
  Since $X_1\subseteq A_1$, we know that $X_1$ is anticomplete to $Y_0$ in $H$.
  Since \(Y_1 \subseteq B_1\), we know that $X_0$ is anticomplete to $Y_1$ in $H$.
  Since $(X_0,Y_0)$ and $(X_1,Y_1)$ satisfy~\ref{itm:good-xy}, we conclude that
  $X_0 \cup X_1$ is anticomplete to $Y_0 \cup Y_1$ in $H$. So this pair satisfies~\ref{itm:good-xy}.

  For the proof of~\ref{itm:good-x}, observe that
  \[|N_H[X_0 \cup X_1] \cap B_0|
  = |N_H[X_0] \cap B_0| + |N_H[X_1] \cap B_1|
  \le \beta|X_0| + \beta|X_1|
  = \beta|X_0 \cup X_1|.\]
  The proof of~\ref{itm:good-y} is symmetric:
  \[|N_H[Y_0 \cup Y_1] \cap A_0|
  = |N_H[Y_0] \cap A_0| + |N_H[Y_1] \cap A_1|
  \le \alpha|Y_0| + \alpha|Y_1|
  = \alpha|Y_0 \cup Y_1|.\]
  This concludes the proof of the claim.
  \cqed\end{proof}

  \begin{claim}\label{clm:small-good}
    If \((A_1, B_1)\) is a pair with \(\emptyset \neq A_1 \subseteq A_0\) and \(B_1 \subseteq B_0\), 
    then there exists a good pair \((X, Y)\) in \((A_1, B_1)\) such that \(X \cup Y \neq \emptyset\), \(|X| \le (\epsi/2)|A|\), and \(|N_H[Y] \cap A_1| \le (\epsi/2)|A|\).
  \end{claim}
  In fact, we shall prove that we can find such a pair $(X,Y)$ with either \(X = \emptyset\) or \(Y = \emptyset\), but this is immaterial for the rest of the proof.
  \begin{proof}
  Suppose to the contrary that there is no pair \((X, Y)\) satisfying the claim.
  Let \(a \in A_1\) be chosen arbitrarily. 
  Let $X^0=\{a\}$.
  For each \(i \in \{0, \dots, k-1\}\),
  %\todo{Fixed radii, please check.}
  % piotrek & michal: checked
  let \[Y^i = N_H^{2i+1}[a] \cap B_1\qquad\textrm{and} \qquad X^{i+1} = N_H^{2i+2}[a] \cap A_1.\]
%  let \(Y^i = N_H[X^i] \cap B_1\), and 
%  let \(X^{i+1} = N_H[Y^i] \cap A_1\). 
  See~Figure~\ref{figs-small-good-pair}.
  %\todo{Replace subscripts with superscripts in the figure.}
  %Piotrek: done

  \begin{figure}[!h]
  \includegraphics{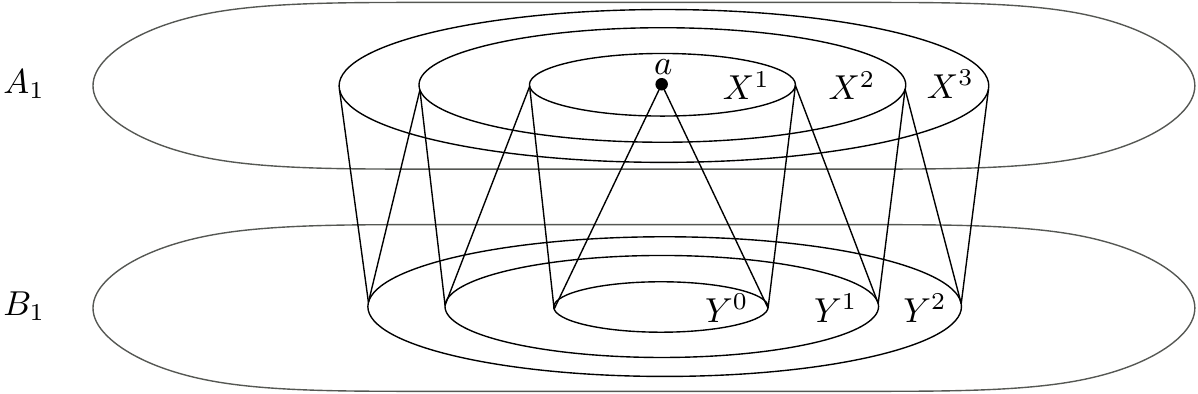}
  \caption{Two sequences of nested sets: $\{a\},X^1,\ldots,X^k$ and $Y^0,Y^1,\ldots,Y^{k-1}$.
  \label{figs-small-good-pair}}
  \end{figure}

  Clearly, $a\in X^i$ and $X^i \subseteq N_H^{2k}[a]\cap A_1$ for every $i\in\{0,\ldots,k\}$.
%  For every \(i \in \{0, \dots, k\}\), we have \(a \in X^i\) and \(X^i \subseteq N_H^{2k}[a] \cap A_0\), and 
  Therefore, by~\eqref{eq:small-balls} we have \(1 \le |X^i| \le (\epsi/2)|A|\). 
  By our assumption the pair \((X^i, \emptyset)\) does not satisfy the claim, so it cannot be good 
  and must therefore break~\ref{itm:good-x}. 
  This way we have
  \[|Y^i| \geq |N_H[X^i] \cap B_1| > \beta \cdot |X^i|,
  \]
  for every $i\in\{0,\ldots,k-1\}$.
  In particular, \(Y^i\) is a nonempty subset of \(N_{H}^{2k-1}[a] \cap B_1\).
  Since \(N_H[Y^i] \cap A_1 \subseteq N_H^{2k}[a] \cap A_0\), 
  by \eqref{eq:small-balls} we have \(1\leq|N_H[Y^i] \cap A_1| \le (\epsi/2)|A|\).
  By our assumption the pair
  \((\emptyset, Y^i)\) does not satisfy the claim, so it cannot be good, 
  and must therefore break~\ref{itm:good-y}.
  This way we have 
  \[|X^{i+1}| \geq |N_H[Y^i] \cap A_1| > \alpha \cdot |Y^i|,
  \]
  for every $i\in\{0,\ldots,k-1\}$.

  Combining the inequalities, we obtain \(|X^{i+1}| > \alpha \beta |X^i|\) for every \(i \in \{0, \dots, k-1\}\).
  Therefore,
  \[|X^k| > {(\alpha\beta)}^k\cdot |X^0| = ((\epsi/2)n^{\epsi})^k\cdot 1  = (\epsi/2)^k n^{\epsi k} \ge (\epsi/2) n \ge (\epsi/2)|A|,\]
  which contradicts~\eqref{eq:small-balls} as $X^k \subseteq N_H^{2k}[a]\cap A_0$.
  \cqed\end{proof}

  \begin{claim}\label{clm:asterisk}
    For every good pair \((X, Y)\) in \((A_0, B_0)\) such that
    \begin{enumorig}[label=\textup{(\roman*)}, start=1]
      \item\label{itm:i} \(|A \cap S| + |X| \le (1-\epsi)|A|\) and
      \item\label{itm:ii} \(|B \cap S| + |Y| \le |B|/n^{\epsi}\),
    \end{enumorig}
    there exists a good pair \((X^*, Y^*)\) in \((A_0, B_0)\) with \(|X^*|+|Y^*| > |X| + |Y|\) such that
    \begin{enumorig}[label=\textup{(\roman*)}, start=3]
      \item\label{itm:iii} \(|A \cap S| + |X^*| \le (1-\epsi/2)|A|\) and
      \item\label{itm:iv} \(|N_H[Y^*] \cap A_0| \le \epsi|A|\).
    \end{enumorig}
\end{claim}
\begin{proof}
  Let \((X, Y)\) be a good pair in \((A_0, B_0)\) that satisfies~\ref{itm:i} and~\ref{itm:ii}. 
  Then by~\ref{itm:good-y} and~\ref{itm:ii} we have
  \begin{equation}
  |N_H[Y] \cap A_0| \le \alpha \cdot |Y| \le \frac{|A|}{|B|}\cdot (\epsi/2)n^{\epsi}\cdot |B|/n^{\epsi}=(\epsi/2)|A|.
  \label{eq:Y0-neighborhood}
  \end{equation}
  Let \(A_1 = A_0 \setminus (X \cup N_{H}[Y])\) and \(B_1 = B_0 \setminus (Y \cup N_{H}[X])\).
  We have
  \begin{align*}
  |A_1|
  &= |A| - (|A \cap S| + |X| + |N_H[Y] \cap A_0|)\\
  &\ge |A| - ((1-\epsi)|A| + (\epsi/2)|A|)&&\text{by~\ref{itm:i} and~\eqref{eq:Y0-neighborhood}}\\
  &= (\epsi/2)|A| > 0.
  \end{align*}
  Hence \(A_1 \neq \emptyset\).
  By Claim~\ref{clm:small-good}, there is a good pair \((X_1, Y_1)\) in \((A_1, B_1)\) with \(X_1 \cup Y_1 \neq \emptyset\) such that 
  \begin{equation}
  |X_1| \le (\epsi/2)|A|
  \label{eq:X1}
  \end{equation} 
  and 
  \begin{equation}
  |N_H[Y_1] \cap A_1| \le (\epsi/2)|A|.
  \label{eq:Y1-neighborhood}
  \end{equation}
  Let \(X^* = X \cup X_1\) and \(Y^* = Y \cup Y_1\).
  By Claim~\ref{clm:good-union}, the pair \((X^*, Y^*)\) is good in \((A_0, B_0)\). 
  Since \(X_1 \cap X\subseteq A_1 \cap X = \emptyset\), \(Y_1 \cap Y \subseteq B_1 \cap Y = \emptyset\), and \(X_1 \cup Y_1 \neq \emptyset\), we have
  \[|X^*| + |Y^*| = |X| + |Y| + |X_1| + |Y_1| > |X| + |Y|.\]

  Finally, we have
  \begin{align*}
  |A \cap S| + |X^*| &= |A \cap S| + |X| + |X_1|\\
  &\le (1-\epsi)|A| + (\epsi/2)|A| = (1-\epsi/2)|A|&&\text{by~\ref{itm:i} and~\eqref{eq:X1}}
  \end{align*}
  and
  \begin{align*}
  |N_H[Y^*] \cap A_0| &\le |N_H[Y] \cap A_0| + |N_H[Y_1] \cap A_1|\\
  &\le (\epsi/2)|A| + (\epsi/2)|A| = \epsi|A|&&\text{by~\eqref{eq:Y0-neighborhood} and~\eqref{eq:Y1-neighborhood},}
  \end{align*}
  as required.
  The claim follows.
\cqed\end{proof}
  We now come back to the proof of Lemma~\ref{lem:weak-abs}.
  If \(|A \cap S| \ge (1 - \epsi)|A|\), then the statement of the lemma is satisfied by taking \(A' = A \cap S\) and \(B' = B\).
  Therefore, we may assume that \(|A \cap S| \le (1 - \epsi)|A| \le (1-\epsi/2)|A|\).
  Hence, the pair \((X^*, Y^*) = (\emptyset, \emptyset)\) (which is a good pair in \((A_0, B_0)\))
  satisfies the conditions~\ref{itm:iii} and~\ref{itm:iv} from Claim~\ref{clm:asterisk}.
  Let us fix a good pair \((X^*, Y^*)\) in \((A_0, B_0)\) satisfying the conditions~\ref{itm:iii} and~\ref{itm:iv} from Claim~\ref{clm:asterisk} which maximizes the value of \(|X^*| + |Y^*|\).
  Since we cannot apply Claim~\ref{clm:asterisk} to increase \(|X^*| + |Y^*|\), the pair \((X^*, Y^*)\) must violate one of the conditions~\ref{itm:i} or~\ref{itm:ii}, so 
  \[|A \cap S| + |X^*| \ge (1-\epsi)|A|\qquad\text{or}\qquad |B \cap S| + |Y^*| \ge |B|/n^{\epsi}.
  \]

  Suppose first that \(|A \cap S| + |X^*| \ge (1-\epsi)|A|\).
  Let \(A' = (A \cap S) \cup X^*\) and \(B' = B \setminus N_H[X^*]\). This way we have \(|A'| = |A \cap S| + |X^*| \ge (1-\epsi)|A|\) and
  \begin{align*}
  |B'|
  &= |B| - |N_H[X^*] \cap B_0|\\
  &\ge |B| - \beta |X^*|&&\text{by~\ref{itm:good-x}}\\
  &\ge |B| - \frac{|B|}{|A|}\cdot (1-\epsi/2)|A|&&\text{by~\ref{itm:iii}}\\
  &= (\epsi/2)|B| \ge |B|/n^{\epsi}&&\text{as $n \ge n_0\geq (2/\epsi)^{1/\epsi}$.}
  \end{align*}
  Since $A'\setminus S=X^*$ and $B'\setminus S \subseteq B_0\setminus N_H[X^*]$,
  we conclude that $A'\setminus S$ is anticomplete to $B'\setminus S$ in~$H$.
  Therefore, the sets $A'$, $B'$, and $S$ satisfy all the properties requested in the lemma statement.

  Now suppose that \(|B \cap S| + |Y^*| \ge |B|/n^{\epsi}\).
  Let \(A' = A \setminus N_H[Y^*]\) and \(B' = (B \cap S) \cup Y^*\).
  This way we have
  \begin{align*}
  |A'| &= |A| - |N_H[Y^*] \cap A_0|\\
  &\ge |A| - \epsi|A| =(1-\epsi)|A|&&\text{by~\ref{itm:iv}.}
  \end{align*}
  and \(|B'| = |B \cap S| + |Y^*| \ge |B|/n^{\epsi}\).
  Again, $A'\setminus S$ is anticomplete to $B'\setminus S$ in $H$, so the sets $A'$,~$B'$, and $S$ satisfy all the properties requested in the lemma statement.
  This concludes the proof of Lemma~\ref{lem:weak-abs}.
\end{proof}

Now we are going to describe a localization argument using which we lift the statement of Lemma~\ref{lem:weak-abs} to obtain Lemma~\ref{lem:new-abs}.

Let $\calF$ be a family of sets, each of size at most $p$.
We say that a subfamily $\calF'$ of $\calF$ \emph{$q$-represents} $\calF$ if for every set \(S\) of size at most \(q\) such that some element of \(\calF\) is disjoint from \(S\), there is an element of \(\calF'\) which also is disjoint from \(S\).
It turns out that we can always find a $q$-representative subfamily of $\calF$ of size at most $\binom{p+q}{q}$.
This is a direct consequence of the following generalization of Bollob\'{a}s's Two Families Theorem:
\begin{lemma}
\label{lem:Bollobas}
Let $A_1,\ldots,A_m$ and $B_1,\ldots,B_m$ be two sequences of sets
such that
$A_i \cap B_j = \emptyset$ if and only if \(i = j\). Then
\[
\sum_{i=1}^m{\binom{a_i + b_i}{b_i}}^{-1} \le 1,
\]
where \(a_i = |A_i|\) and \(b_i = |B_i|\).
\end{lemma}

We refer the reader for instance to~\cite{Jukna-book} for a very elegant proof of this statement.
In the sequel, we will use the following.

\begin{corollary}
\label{cor-representative-sets}
  Let $p$ and $q$ be positive integers and 
  let $\calF$ be a family of sets, each of size at most~$p$.
  Then there is a subfamily $\calF'$ of $\calF$ with $|\calF'| \leq \binom{p + q}{p}$ that 
  $q$-represents $\calF$. 
\end{corollary}
\begin{proof}
Consider a minimal subfamily $\calF'$ of $\calF$ which
$q$-represents $\calF$.
Clearly, $\calF'$ is well-defined as $\calF$ does $q$-represent itself.
We enumerate the elements of $\calF'$ as $A_1,\ldots,A_m$ where $m=|\calF'|$.
Since $\calF'$ is minimal, for every set $A_i$ there exists a set $B_i$ of size at most \(q\) such that
 $A_i\cap B_i=\emptyset$ and $A_j\cap B_i\neq\emptyset$ for $j\neq i$.
Therefore, the sequences \(A_1, \dots, A_m\) and \(B_1, \dots, B_m\)
satisfy the assumptions of Lemma~\ref{lem:Bollobas} and we conclude that $\sum_{i=1}^m{\binom{a_i + b_i}{b_i}}^{-1} \le 1$, where \(a_i = |A_i|\) and \(b_i = |B_i|\). Since \(a_i \le p\) and \(b_i \le q\),
we have \(\binom{a_i + b_i}{b_i} \le \binom{p+q}{q}\).
Therefore \(m \cdot {\binom{p+q}{q}}^{-1} \le 1\), so \(|\calF| = m \le \binom{p+q}{q}\).
\end{proof}

\begin{proof}[Proof of Lemma~\ref{lem:new-abs}]
It suffices to prove the lemma for sufficiently small values of \(\epsi\), so let
us assume \(\epsi < 1/2\).
Let \(\epsi'\) be a positive real satisfying
\[\epsi' < \frac{\epsi}{2 + (d+1)\epsi}.\]
Let \(\calC'\) be the closure of \(\calC\) under taking subgraphs.
Since \(\calC\) is nowhere dense, so is \(\calC'\). 
Let \(n_0'\) be the constant obtained by an application of Lemma~\ref{lem:weak-abs} to \(\calC'\), \(d\), and \(\epsi'\).

Let \(k_1\) be an integer such that for every \(k \ge k_1\) we have 
\[
k + k^2(d-1)\binom{\floor{k^{\epsi}}+d+1}{d+1} \le {(k^{\epsi}/2)}^{1/\epsi'}.\] 
We can find such integer because the left-hand side of the inequality is \(\Oh(k^{2+(d+1)\epsi})\), the right-hand side is \(\Omega(k^{\epsi/\epsi'})\) and \(2+(d+1)\epsi < \epsi/\epsi'\).
Moreover, as \(\epsi < 1/2\), we can choose an integer \(k_2\) such that for
every \(k \ge k_2\) we have
\[k^{1-\epsi} - 2 \ge k^{\epsi}.\]
Let
\[k_0 = \max\left\{n_0', k_1, \ceil*{{(\epsi - \epsi')}^{-1/\epsi}}, k_2\right\}.\]
Let \(G\) be a graph in  \(\calC\) and let \(A\) and \(B\) be two vertex subsets in \(G\) with \(|A \cup B| \ge k_0\). We aim to show the existence of sets \(A'\), \(B'\), and \(S\) as in the statement of the lemma.

Let \[k=|A \cup B|,\qquad p = d + 1,\qquad\textrm{and}\qquad q = \floor{{k}^{\epsi}}.\]
For \(a \in A\) and \(b \in B\),
consider all paths in \(G\) of length at most \(d\) which have one endpoint in \(a\) and
second in \(b\), and 
let \(\calF_{a, b}\) denote the family comprising the vertex sets of such paths.
Let $\calF'_{a,b}$ be a subfamily of \(\calF_{a, b}\) that $q$-represents \(\calF_{a, b}\).
By Corollary~\ref{cor-representative-sets}, we can choose $\calF'_{a,b}$ so that  \(|\calF'_{a,b}|\leq\binom{p+q}{p}\).
Let \(G'\) be the subgraph of \(G\) induced by the vertices in the set
\[
A \cup B\cup
\bigcup_{\substack{a\in A, b\in B\\X\in\calF'_{a, b}}}X.
\]
Since each set in a family \(\calF'_{a, b}\) has at most \(d-1\) vertices outside \(A \cup B\), we have
\begin{align*}
|V(G')| &= |A \cup B| + \sum_{a \in A, b \in B} \sum_{X \in \calF'_{a, b}} |X \setminus (A \cup B)|\\
&\le k + k^2(d-1)\binom{\floor{k^{\epsi}} + d + 1}{d+1} \\
&\le {(k^{\epsi}/2)}^{1/\epsi'}&&\text{as } k \ge k_0 \ge k_1.
\end{align*}
Thus \(|V(G')|^{\epsi'} \le k^{\epsi}/2\).

Since \(G'\in\calC'\) and \(|V(G')| \ge |A \cup B| = k \ge k_0 \ge n_0'\),
by Lemma~\ref{lem:weak-abs} we can obtain sets
\(A'\), \(B'\), and \(S\) with \(A' \subseteq A\), \(B' \subseteq B\), and \(S \subseteq V(G')\)
such that \(|A'| \ge (1-\epsi')|A|\), \(|B'| \ge |B|/|V(G')|^{\epsi'} \ge 2|B|/k^{\epsi}\), \(|S| \le |V(G')|^{\epsi'} \le k^{\epsi}/2 \le k^{\epsi}\), and
\(S\) \(d\)-separates \(A'\) and \(B'\) in \(G'\).
We claim that \(S\) \(d\)-separates \(A'\) and \(B'\) in the graph \(G\) as well.
Suppose to the contrary that in \(G\) there is a path
\(P\) of length at most \(d\) with one endpoint in \(A'\) and second in \(B'\)
such that \(P\) does not contain any vertex from \(S\).
Let \(a\) denote the endpoint of \(P\) in \(A'\) and let \(b\) denote
the endpoint of \(P\) in \(B'\).
The set \(V(P)\) is a member of \(\calF_{a, b}\) that is disjoint from the set \(S\). Since \(|S|\leq q\) and
 \(\calF'_{a, b}\) \(q\)-represents the family \(\calF_{a, b}\), there exists a set \(X \in \calF'_{a, b}\) which is disjoint from
\(S\).
Then the set \(X\) is a vertex set of a path of length at most \(d\) connecting \(A'\) and \(B'\) in \(G'\) which is disjoint from \(S\), a contradiction.

It remains to argue that we can modify the sets \(A'\) and \(B'\) by removing some elements so that they
become disjoint and still remain sufficiently large.
Note that currently all common elements of \(A'\) and \(B'\) must lie in \(S\) because \(S\) \(d\)-separates \(A'\) and \(B'\).
We aim to remove every element of \(A' \cap B'\) from either \(A'\) or \(B'\).
Note that \(|A'| \ge (1-\epsi')|A| = (1 - \epsi)|A| + (\epsi - \epsi')|A|\),
so we have a surplus of at least \(\floor*{(\epsi - \epsi')|A|}\) vertices in \(A'\) which we
can remove without violating the \(|A'| \ge (1-\epsi)|A|\) requirement.
We also have \(|B'| \ge 2|B|/k^{\epsi} = {|B|/k^{\epsi}} + |B|/k^{\epsi}\),
so we can remove at least \(\floor*{|B|/k^{\epsi}}\) vertices from \(B'\).
Thus, the total number of vertices we can remove from sets \(A'\) and \(B'\) is
at least \((\epsi - \epsi')|A| - 1 + |B|/k^{\epsi} - 1\), and
\begin{align*}
(\epsi - \epsi')|A| - 1 + |B|/k^{\epsi} - 1 &\ge |A|/k^{\epsi} - 1 + |B|/k^{\epsi}-1&&\text{as }k \ge k_0 \ge {(\epsi - \epsi')}^{-1/\epsi}\\
& \ge |A \cup B|/k^\epsi - 2\\
& = k^{1-\epsi} - 2\\
& \ge k^{\epsi}&&\text{as }k \ge k_0 \ge k_2\\
& \ge |S|.
\end{align*}
Therefore, after removing each element of \(A' \cap B'\) from either \(A'\) or \(B'\),
the sets \(A'\) and \(B'\) satisfy properties requested in the lemma statement.
\end{proof}

%=================================================================
\section{Proof of Theorem~\ref{thm:main}}
\label{sec:proofs-of-the-theorem}

Let $G$ be a graph, $S$ be a set of vertices in \(G\), and let \(d\) be a positive integer.
For a vertex $u$ in $G$, the \emph{distance-$d$ profile} of $u$ on $S$ in $G$ is defined as a function $\profile^d_G[u, S]$ from $S$ to $\set{0, \dots, d, \infty}$ given by
\[
\profile^d_G[u, S](x) =
\begin{cases}
  \dist_G(u, x) &\text{if $\dist_G(u, x) \leq d$,}\\
  \infty &\text{otherwise.}
\end{cases}
\] 
The \emph{distance-$d$ profile complexity} of $S$ in $G$ is the number of different distance-$d$ profiles on $S$ in~$G$.
We will use the following result of Eickmeyer et al.~\cite{EickmeyerGKKPRS17}.

\begin{theorem}[\cite{EickmeyerGKKPRS17}]\label{thm:distance-profile-complexity}
  Let $\calC$ be a nowhere dense class of graphs, \(d\) be a positive integer, 
  and $\epsi$ be a positive real.
  Then there is an integer \(m_0\)
  such that
  for every graph $G$ in $\calC$ and every subset $S$ of $V(G)$ with \(|S| \ge m_0\),
  the distance-\(d\) profile complexity of \(S\) in $G$ is at most $|S|^{1+\epsi}$.
\end{theorem}

With all the tools prepared, we may proceed to the proof of our main result.

\begin{proof}[Proof of Theorem \ref{thm:main}]
  It is enough to prove the theorem for sufficiently small \(\epsi\), so
  we assume that \(\epsi < 1\).
  Let \(\epsi' = \epsi / 3\).
  Let \(k_0'\) be the integer obtained by
  applying Lemma~\ref{lem:new-abs} to \(\calC\), \(d\), and \(\epsi'\).
  Let \(m_0\) be the integer obtained by applying Theorem~\ref{thm:distance-profile-complexity}
  to \(\calC\), \(d\) and \(\epsi=1\), so that the distance-\(d\) profile complexity of
  any \(m\)-element vertex subset in a graph from \(\calC\) is at most \(m^2\), provided \(m \ge m_0\).
  We prove the theorem for \(k_0\) defined as
  \[k_0 = \max \left\{k_0', \ceil*{(m_0 + 1)^{1/\epsi'}}\right\}.\]
  
  Let \(G\) be a graph in \(\calC\) and
  let \(A\) and \(B\) be two nonempty sets of vertices in \(G\) satisfying \(|A \cup B| \geq k_0\).
  Let \(k=|A \cup B|\).
  Since \(k\geq k_0 \ge k_0'\),
  there exist disjoint subsets \(A_0\) and \(B_0\)
  of \(A\) and \(B\) respectively,
  and there is a vertex subset \(S\) in \(G\) such
  that \(|A_0| \ge (1-\epsi')|A| \ge (1-\epsi)|A|\), \(|B_0| \ge |B|/k^{\epsi'}\),
  \(|S| \le k^{\epsi'}\), and every \(A_0\)--\(B_0\)-path of length at most
  \(d\) in \(G\) intersects \(S\).
  By adding arbitrary vertices of $G$ if necessary, we may assume that
  $k^{\epsi'}\geq |S| > k^{\epsi'}-1$. This is possible since $\epsi'\le 1$ and $|V(G)|\geq |A\cup B|= k$.
  Therefore,
  \[
  k^{\epsi'} \geq |S| > k^{\epsi'}-1 \ge k_0^{\epsi'}-1 \ge (m_0 + 1) - 1 = m_0.
  \]
  We infer that the distance-$d$ profile complexity of \(S\) in \(G\) is at most \(|S|^2\).
  Therefore, there is a subset \(B'\) of
%  Since \(|S|^2 \le k^{2\epsi'}\), 
  \(B_0\) in which all vertices have the same distance-\(d\) profile on \(S\)
  and
  \[
  |B'| \ge |B_0|/|S|^2 \geq |B_0|/k^{2\epsi'} \ge |B|/k^{3\epsi'} = |B|/k^{\epsi}.
  \]

  To prove that the sets \(A_0\) and \(B'\) satisfy the theorem, it suffices to
  show that all elements of \(B'\) have the same distance-\(d\) profile on \(A_0\),
  that is for any \(a \in A_0\) and
  \(b_1, b_2 \in B'\),  if \(\dist_G(b_1, a) \le d\) then
  \(\dist_G(b_2, a) = \dist_G(b_1, a)\).
  Suppose then, that \(\dist_G(b_1, a) \le d\).
  A shortest \(a\)--\(b_1\) path intersects
  \(S\) in a vertex \(s\) such that
  \(\dist_G(b_1, a) = \dist_G(b_1, s) + \dist_G(s, a)\).
  Since \(\profile^d_G[b_1, S] = \profile^d_G[b_2, S]\),
  we have
  \[\dist_G(b_2, a) \le \dist_G(b_2, s) + \dist_G(s, a) =
  \dist_G(b_1, s) + \dist_G(s, a) = \dist_G(b_1,a) \le d.\]
  A symmetric argument shows that \(\dist_G(a, b_1) \le \dist_G(a, b_2)\), so
  indeed \(\dist_G(a, b_1) = \dist_G(a, b_2)\).
  This completes the proof of the theorem.
\end{proof}

\newcommand{\tup}[1]{\bar{#1}}
\newcommand{\adj}{\mathsf{adj}}

\section{From powers to interpretations}\label{sec:interpretations}

We now use the results obtained in the previous sections to prove a generalization to the setting of classes obtained from nowhere dense ones using (one-dimensional) First Order interpretations. 
That is, we prove Theorem~\ref{thm:interpretations}.
Note that the statement of this result does not involve the notion of $\FO$ interpretations explicitly, but a reader familiar with this concept should see a connection.
We refer to the work of Gajarsk\'y et al.~\cite{GajarskyKNMPST18} for a broader discussion of $\FO$ interpretations and their relation with the notions of sparsity.

\medskip

We assume basic familiarity with the First Order logic ($\FO$) on undirected graphs. 
Recall that in this logic, variables correspond to single vertices and formulas are built recursively from smaller subformulas using the following constructs: existential and universal quantification, and standard boolean
connectives. The basic ({\em{atomic}}) formulas are of the form $x=y$ and $\adj(x,y)$; they respectively check whether vertices assigned to variables $x$ and $y$ are equal or adjacent.
For instance, for a nonnegative integer $d$, the following inductive definition gives a formula $\delta_d(x,y)$ such that $\delta_d(u,v)$ is true in a graph $G$ if and only if \(\dist_G(u, v) \le d\).
\begin{eqnarray*}
 \delta_0(x,y) & \coloneqq & (x=y)\\
 \delta_d(x,y) & \coloneqq & \delta_{d-1}(x,y)\vee \\
 & & \left[\exists_{z_1}\ldots\exists_{z_{d-1}}\ \adj(x,z_1)\wedge \adj(z_1,z_2)\wedge\ldots \wedge \adj(z_{d-2},z_{d-1})\wedge \adj(z_{d-1},y)\right].
\end{eqnarray*}
If $\varphi(\tup x)$ is a formula with a tuple of free variables $\tup x$ and $\tup u\colon \tup x\to V(G)$ is an evaluation of variables of $\tup x$ in a graph $G$, then we write $G\models \varphi(\tup u)$ to
denote that $\varphi$ holds in $G$ when each variable of $\tup x$ is evaluated as prescribed by $\tup u$.

The main tool that will replace results on neighborhood complexity (Theorem~\ref{thm:distance-profile-complexity}) are the bounds on the number of {\em{$\FO$-types}} in nowhere dense classes, 
which were recently proved by Pilipczuk et al.~\cite{PilipczukST18a}.
We refer to this work for a broader discussion, while we will rely on the following corollary, which follows easily from~\cite[Theorem~1.3 and Lemma~3.1]{PilipczukST18a}.

\begin{lemma}[\cite{PilipczukST18a}]\label{lem:types}
Let $\Cc$ be a nowhere dense class of graphs and $\varphi(x,y)$ be an $\FO$ formula with two free variables. 
Then there exist positive integers $c$ and $d$ such that for every graph $G$ in $\Cc$ and vertex subset $S\subseteq V(G)$, there exists a coloring $\lambda$ of vertices of $G$ with at most $c\cdot |S|^c$ colors such that the following holds:
for every triple of vertices $u,v,v'\in V(G)$ such that $\lambda(v)=\lambda(v')$ and $S$ intersects every path of length at most $d$ connecting $u$ with any of $\{v,v'\}$, we have
$$G\models \varphi(u,v)\qquad\textrm{if and only if}\qquad G\models \varphi(u,v').$$
\end{lemma}

\medskip 

We are now ready to prove Theorem~\ref{thm:interpretations}.

\begin{proof}[Proof of Theorem~\ref{thm:interpretations}]
 Let $c$ and $d$ be the constants given by Lemma~\ref{lem:types} for the class $\Cc$ and formula $\varphi(x,y)$. Let $k=|A\cup B|$ and let $\eps'$ be a positive real, to be determined later.
 Assuming that $k$ is large enough depending on $\Cc$, $d$, and $\eps'$, we may apply Lemma~\ref{lem:new-abs} to $A$ and $B$ and constants $d$ and~$\eps'$. 
 Thus, we find disjoint subsets $A_0$ and $B_0$ of $A$ and $B$, respectively, and a vertex subset $S$ such that
 \begin{itemize}
  \item $|A_0|\geq (1-\eps')|A|$, $|B_0|\geq |B|/k^{\eps'}$, and $|S|\leq k^{\eps'}$; and
  \item for every $a\in A_0$ and $b\in B_0$, every path of length at most $d$ connecting $a$ and $b$ intersects~$S$.
 \end{itemize}
 Now let $\lambda$ be the coloring of vertices of $G$ provided by Lemma~\ref{lem:types}.
 Then $\lambda$ uses at most $c\cdot |S|^c\leq c\cdot k^{c\eps'}$ colors, hence there is a subset $B'\subseteq B_0$ such that $|B'|\geq \frac{|B|}{c\cdot k^{(c+1)\eps'} }$ and $B'$ is monochromatic in $\lambda$.
 By Lemma~\ref{lem:types} and the second point above, for every $a\in A_0$ we either have $G\models \varphi(a,b)$ for all $b\in B'$, or $G\models \neg \varphi(a,b)$ for all $b\in B'$.
 Hence, we may set $A'=A_0$.
 It remains to note that by choosing $k_0$ and $\eps'<\frac{\eps}{c+1}$ appropriately, we have $|A'|\geq \left(1-\eps\right) |A|$ and $|B'|\geq |B|/k^{\eps}$.
\end{proof}

\section{Negative example}
\label{sec:example}
%    \begin{example}
%        There exists a class of graphs \(\mathcal{C}\) which is nowhere dense but does not have the strong Erd\H{o}s--Hajnal property.
%    \end{example}
    \begin{proof}[Proof of Proposition~\ref{prop:no-strong-EH}]
        Consider the \emph{Erd\H{o}s-R\'enyi random graph model} --- for fixed \(n \in \mathbb{N}\) and \(p \in [0,1]\) it is the distribution \(\mathcal{G}_{n,p}\) on \(n\)-vertex graphs, in which every edge is present with probability \(p\) independently from other edges. We wish to use this distribution to obtain our desired family \(\mathcal{C}\). Let us fix an integer \(g \geq 3\) and consider the distribution \(\mathcal{G}_{n,p}\) for some large value of \(n\) determined later and \(p = n^{\frac{1}{2g}-1}\).
        Denote by \(X\) the random variable counting the number of cycles of length at most \(g\) in a graph drawn from \(\mathcal{G}_{n,p}\). 
        We easily obtain, for \(n\) large enough, the following bound on the expected value of $X$:
        \[
            \mathbb{E}(X) = \sum_{k = 3}^{g} {n \choose k} k! \frac{p^k}{2k} \leq \sum_{k = 3}^{g} \frac{(np)^k}{2k} \leq (np)^g = \sqrt{n}.
        \]
        By Markov's inequality, the probability that \(X \geq 3 \sqrt{n}\) is at most \(1/3\).
        Let us the following event \(A\): a graph drawn from \(\mathcal{G}_{n,p}\) contains a pair of disjoint vertex subsets which are either complete or anticomplete to each other, and are of size at least \(n / g\).
        Using the union bound we obtain the following upper bound on the probability of $A$:
        \begin{align*}
            \mathbb{P}(A) \leq\ & 2^{n} \cdot 2^n  \cdot \left( p^{(n/g)^2} + (1-p)^{(n/g)^2} \right) = 2^{2n} \cdot n^{\left(\frac{1}{2g}-1\right)(n/g)^2} + 2^{2n} \cdot \left(1 - \frac{1}{n^{1 - \frac{1}{2g}}}\right)^{(n/g)^2}. %\\
                            % =\  & 2^{2n} \cdot p^{(n/g)^2} + 2^{2n} \cdot \left(1 - \frac{1}{n^{1 - \frac{1}{2g}}} \right)^{n^{1 - \frac{1}{2g}} \cdot n^{1 + \frac{1}{2g}} /g^2}
        \end{align*}
        Both of the summands on the right hand side tend to \(0\) as \(n \to \infty\). So let us assume that \(n\) is large enough so that \(\mathbb{P}(A)<1/3\).

        Therefore, for \(n\) large enough, the probability that a graph drawn from \(\mathcal{G}_{n,p}\) has less than \(3 \sqrt{n}\) cycles of length at most \(g\) 
        and contains no pair of complete/anticomplete vertex subsets of size \(n / g\) is positive. In particular, a graph enjoying both these properties exists. Let us call it \(H_g\).

        Now consider every cycle of length at most \(g\) in \(H_g\), and let us call \(G_g\) the graph obtained from \(H_g\) by deleting one vertex from each such cycle. As there are less than \(3 \sqrt{n}\) of these cycles, $G_g$ has at least \(n/2\) vertices. Moreover, deleting vertices cannot create a pair of complete or anticomplete sets of a given size where none existed.
        Now we can define the following family of graphs
        \[
            \mathcal{C} = \set*{\,G_g \ \colon \ g \in \mathbb{N},\, g \geq 3\,}.
        \]
        We argue that (1) \(\mathcal{C}\) is nowhere dense, and 
        (2) \(\mathcal{C}\) does not satisfy the strong Erd\H{o}s--Hajnal property.

        Since the sizes of graphs in \(\mathcal{C}\) are bounded by a function of their girth, $\mathcal{C}$ is indeed nowhere dense. By the construction, \(G_g\) does not contain a complete/anticomplete pair of sets of size  \(\frac{n}{g} \leq \frac{2|V(G_g)|}{g}\), so the strong Erd\H{o}s--Hajnal property does not hold.
    \end{proof}

 \bibliographystyle{plain}
 \bibliography{sparsity}

\begin{thebibliography}{10}

\bibitem{AdlerA14}
Hans Adler and Isolde Adler.
\newblock Interpreting nowhere dense graph classes as a classical notion of
  model theory.
\newblock {\em European Journal of Combinatorics}, 36:322--330, 2014.

\bibitem{APPRS05}
Noga Alon, János Pach, Rom Pinchasi, Radoš Radoičić, and Micha Sharir.
\newblock Crossing patterns of semi-algebraic sets.
\newblock {\em Journal of Combinatorial Theory, Series A}, 111(2):310 -- 326,
  2005.

\bibitem{Chudnovsky14}
Maria Chudnovsky.
\newblock The {E}rd{\H{o}}s-{H}ajnal conjecture --- {A} survey.
\newblock {\em Journal of Graph Theory}, 75(2):178--190, 2014.

\bibitem{DKT12}
Zden\v{e}k Dvo\v{r}\'{a}k, Daniel Kr\'{a}l, and Robin Thomas.
\newblock Testing first-order properties for subclasses of sparse graphs.
\newblock {\em Journal of the ACM}, 60(5), 2013.

\bibitem{EickmeyerGKKPRS17}
Kord Eickmeyer, Archontia~C. Giannopoulou, Stephan Kreutzer, O{-}joung Kwon,
  Micha\l{} Pilipczuk, Roman Rabinovich, and Sebastian Siebertz.
\newblock Neighborhood complexity and kernelization for nowhere dense classes
  of graphs.
\newblock In {\em 44th International Colloquium on Automata, Languages, and
  Programming, {ICALP} 2017}, volume~80 of {\em LIPIcs}, pages 63:1--63:14.
  Schloss Dagstuhl --- Leibniz-Zentrum f{\"{u}}r Informatik, 2017.

\bibitem{EH89}
Paul Erdős and Andr\'{a}s Hajnal.
\newblock Ramsey-type theorems.
\newblock {\em Discrete Applied Mathematics}, 25(1):37 -- 52, 1989.

\bibitem{FP08}
Jacob Fox and J{\'a}nos Pach.
\newblock {\em {E}rd{\H{o}}s-{H}ajnal-type Results on Intersection Patterns of
  Geometric Objects}, pages 79--103.
\newblock Springer Berlin Heidelberg, 2008.

\bibitem{FPT10}
Jacob Fox, J{\'a}nos Pach, and Csaba~D. T{\'o}th.
\newblock Tur{\'a}n-type results for partial orders and intersection graphs of
  convex sets.
\newblock {\em Israel Journal of Mathematics}, 178(1):29--50, Sep 2010.

\bibitem{GajarskyKNMPST18}
Jakub Gajarsk{\'{y}}, Stephan Kreutzer, Jaroslav Ne\v{s}et\v{r}il, Patrice
  {Ossona de Mendez}, Micha\l{} Pilipczuk, Sebastian Siebertz, and Szymon
  Toru{\'{n}}czyk.
\newblock First-order interpretations of bounded expansion classes.
\newblock In {\em 45th International Colloquium on Automata, Languages, and
  Programming, {ICALP} 2018}, volume 107 of {\em LIPIcs}, pages 126:1--126:14.
  Schloss Dagstuhl --- Leibniz-Zentrum f{\"{u}}r Informatik, 2018.

\bibitem{GanianHNOM19}
Robert Ganian, Petr Hlin\v{e}n{\'{y}}, Jaroslav Ne\v{s}et\v{r}il, Jan
  Obdr\v{z}{\'{a}}lek, and Patrice {Ossona de Mendez}.
\newblock Shrub-depth: {C}apturing height of dense graphs.
\newblock {\em Logical Methods in Computer Science}, 15(1), 2019.

\bibitem{GroheKS17}
Martin Grohe, Stephan Kreutzer, and Sebastian Siebertz.
\newblock Deciding first-order properties of nowhere dense graphs.
\newblock {\em Journal of the ACM}, 64(3):17:1--17:32, 2017.

\bibitem{Har-PeledQ17}
Sariel Har{-}Peled and Kent Quanrud.
\newblock Approximation algorithms for polynomial-expansion and low-density
  graphs.
\newblock {\em {SIAM} Journal on Computing}, 46(6):1712--1744, 2017.

\bibitem{JNOdMS20}
Yiting Jiang, Jaroslav Nešetřil, Patrice~Ossona de~Mendez, and Sebastian
  Siebertz.
\newblock Regular partitions of gentle graphs.
\newblock {\em CoRR}, abs/2003.11692, 2020.

\bibitem{Jukna-book}
Stasys Jukna.
\newblock {\em Extremal Combinatorics - With Applications in Computer Science}.
\newblock Texts in Theoretical Computer Science. An {EATCS} Series. Springer,
  2011.

\bibitem{KPS20}
O{-}joung Kwon, Michał Pilipczuk, and Sebastian Siebertz.
\newblock On low rank-width colorings.
\newblock {\em European Journal of Combinatorics}, 83:103002, 2020.

\bibitem{MalliarisSh14}
Maryanthe Malliaris and Saharon Shelah.
\newblock Regularity lemmas for stable graphs.
\newblock {\em Transactions of American Mathematical Society}, 366:1551--1585,
  2014.

\bibitem{NOdMPZ19}
Jaroslav Ne\v{s}et\v{r}il, Patrice~Ossona de~Mendez, Micha\l{} Pilipczuk, and
  Xuding Zhu.
\newblock Clustering powers of sparse graphs.
\newblock {\em CoRR}, abs/2003.03605, 2020.

\bibitem{NesetrilM08a}
Jaroslav Ne\v{s}et\v{r}il and Patrice {Ossona de Mendez}.
\newblock Grad and classes with bounded expansion {I}. {D}ecompositions.
\newblock {\em European Journal of Combinatorics}, 29(3):760--776, 2008.

\bibitem{NesetrilM08b}
Jaroslav Ne\v{s}et\v{r}il and Patrice {Ossona de Mendez}.
\newblock Grad and classes with bounded expansion {II.} {A}lgorithmic aspects.
\newblock {\em European Journal of Combinatorics}, 29(3):777--791, 2008.

\bibitem{NesetrilM11}
Jaroslav Ne\v{s}et\v{r}il and Patrice {Ossona de Mendez}.
\newblock On nowhere dense graphs.
\newblock {\em European Journal of Combinatorics}, 32(4):600--617, 2011.

\bibitem{sparsity}
Jaroslav Ne\v{s}et\v{r}il and Patrice {Ossona de Mendez}.
\newblock {\em Sparsity --- {G}raphs, {S}tructures, and {A}lgorithms},
  volume~28 of {\em Algorithms and {C}ombinatorics}.
\newblock Springer, 2012.

\bibitem{NOdMW12}
Jaroslav Nešetřil, Patrice~Ossona de~Mendez, and David~R. Wood.
\newblock Characterisations and examples of graph classes with bounded
  expansion.
\newblock {\em European Journal of Combinatorics}, 33(3):350 -- 373, 2012.

\bibitem{notes}
Marcin Pilipczuk, Micha\l{} Pilipczuk, and Sebastian Siebertz.
\newblock Lecture notes for the course ``{S}parsity'' given at {F}aculty of
  {M}athematics, {I}nformatics, and {M}echanics of the {U}niversity of
  {W}arsaw, Winter semesters 2017/18 and 2019/20.
\newblock Available \url{https://www.mimuw.edu.pl/~mp248287/sparsity2}.

\bibitem{PilipczukST18a}
Micha\l{} Pilipczuk, Sebastian Siebertz, and Szymon Toru{\'{n}}czyk.
\newblock On the number of types in sparse graphs.
\newblock In {\em 33rd Annual {ACM/IEEE} Symposium on Logic in Computer
  Science, {LICS} 2018}, pages 799--808. {ACM}, 2018.

\bibitem{ScottSS19}
Alex Scott, Paul~D. Seymour, and Sophie Spirkl.
\newblock Pure pairs. {V}. {E}xcluding some long subdivision, 2019.
\newblock Manuscript.

\bibitem{Zhu09}
Xuding {Zhu}.
\newblock {Colouring graphs with bounded generalized colouring number.}
\newblock {\em {Discrete Mathematics}}, 309(18):5562--5568, 2009.

\end{thebibliography}

\appendix

\section{Low shrubdepth colorings and the strong Erd\H{o}s-Hajnal property}\label{app:shrubs}

We start by considering the setting when the graph in question admits a connection model of bounded depth and using a bounded number of labels. Here, by the {\em{depth}} of a rooted tree we mean the maximum number of vertices on a root-to-leaf path.

\begin{lemma}\label{lem:low-shrubdepth}
 Let $G$ be a graph on $n\geq 2$ vertices that admits a connection model of depth $d$ and using a label set of size $s$.
 Then there exist disjoint subsets of vertices $A$, $B$ of $G$ such that $$\min(|A|,|B|)\geq \frac{n}{3s\cdot 2^{d-2}}$$ and $A$ is either complete or anticomplete to $B$.
\end{lemma}
\begin{proof}
 Let \((\lambda, T, \set{Z_x \colon x \in V(T) \setminus V(G)})\) be the assumed connection model of $G$. 
 Recall $T$ has depth $d$ and $Z_x\subseteq \Lambda^2$ for some label set $\Lambda$ of size $s$, for each $x \in V(T) \setminus V(G)$.
 Note that since $n\geq 2$, $T$~cannot consist of only one leaf, hence $d\geq 2$.
 We proceed by induction on $d$, where the base case $d=2$ will be argued directly within the reasoning.
 
 Let $r$ be the root of $T$ and let $C$ be the set of children of $r$ in $T$. 
 For every $c\in C$, let $L_c$ comprise the leaves of $T$ that are descendants of $c$ (possibly $L_c=\{c\}$ if $c$ itself is a leaf).
 Then $\{L_c\colon c\in C\}$ is a partition of the vertex set of $G$. We consider two cases.
 
 First, suppose that there exists $c_0\in C$ such that $|L_{c_0}|>n/2$.
 Note here that this case cannot happen for $d=2$, because then $T$ is a star with all the children of $r$ being leaves, hence $|L_c|=1\leq n/2$ for all $c\in C$.
 Therefore, we may apply the induction hypothesis to the graph $G[L_{c_0}]$, which has $|L_{c_0}|>n/2\geq 1$ vertices and admits a connection model of depth $d-1$ and using a label set of size $s$, obtained by restricting the connection model \((\lambda, T, \set{Z_x \colon x \in V(T) \setminus V(G)})\) in a natural way.
Thus, we obtain $A,B\subseteq L_{c_0}$ such that $A$ is either complete or anticomplete to~$B$, and
 $$\min(|A|,|B|)\geq \frac{|L_{c_0}|}{3s\cdot 2^{d-3}}>\frac{n}{3s\cdot 2^{d-2}},$$
 as required.
 
 Second, suppose that for all $c\in C$, we have $|L_c|\leq n/2$.
 It is well-known that then there is a partition of $C$ into $X$ and $Y$ such that for each $Z\in \{X,Y\}$, we have
 $$n/3\leq \sum_{c\in Z} |L_c|\leq 2n/3.$$
 Since the connection model uses a label set $\Lambda$ of size $s$, we can find a set $A$ such that
  $A\subseteq \bigcup_{c\in X} L_c$, $|A|\geq \frac{n}{3s}$, and all vertices of $A$ receive the same label under $\lambda$.
 We also find $B\subseteq \bigcup_{c\in Y} L_c$ such that
  $B\subseteq \bigcup_{c\in Y} L_c$, $|B|\geq \frac{n}{3s}$, and all vertices of $B$ receive the same label under $\lambda$.
 Thus, $A$ and $B$ are disjoint and $$\min(|A|,|B|)\geq \frac{n}{3s}\geq \frac{n}{3s\cdot 2^{d-2}}.$$
 Finally, let $\lambda_A$ be the value of $\lambda(a)$ for all $a\in A$ 
 and $\lambda_B$ be the value of $\lambda(b)$ for all $b\in B$.
 If $(\lambda_A,\lambda_B) \not \in Z_r$ then $A$ is anticomplete to $B$.
 Otherwise, $(\lambda_A,\lambda_B)  \in Z_r$ and then $A$ is complete to~$B$.
 Note that since in this case we have found suitable sets $A$ and $B$ directly, the base case $d=2$ is covered.
\end{proof}

Now, from Lemma~\ref{lem:low-shrubdepth} it is easy to conclude the claimed connection.

\begin{theorem}\label{thm:lsc-EH}
 Let $\Cc$ be a class of graphs that admits low shrubdepth colorings. Then $\Cc$ has the strong Erd\H{o}s-Hajnal property.
\end{theorem}
\begin{proof}
 Consider any graph $G\in \Cc$, say on $n$ vertices. 
 Since we want to verify the strong Erd\H{o}s-Hajnal property of $G$, we assume that $n\geq2$.
 By assumption, there exist numbers $s=s(1)$ and $f=f(1)$ such that $G$ admits a coloring using $f$ colors where each color induces a subgraph of shrubdepth at most $s$.
 By considering the color class of the largest cardinality, we infer that $G$ contains an induced subgraph $H$ such that $H$ has at least $n/f$ vertices and shrubdepth at most~$s$.
 Assuming that $n/f\geq 2$, we may apply Lemma~\ref{lem:low-shrubdepth} to $H$. This yields a pair of disjoint vertex subsets $A$ and $B$ such that $A$ is either complete or anticomplete to $B$, and
 $$\min(|A|,|B|)\geq \frac{|V(H)|}{3s\cdot 2^{s-2}}\geq \frac{n}{3sf\cdot 2^{s-2}}.$$
 Finally, if $n/f<2$, then $n<2f$ and we may choose $A$ and $B$ to be any disjoint singleton sets ($n\geq2$); then $\min(|A|,|B|)>\frac{n}{2f}$.
 We conclude that $\Cc$ has the strong Erd\H{o}s-Hajnal property for $\delta=\frac{1}{3sf\cdot 2^{s-2}}$.
\end{proof}

By combining Theorem~\ref{thm:lsc-EH} with the results of Kwon et al.~\cite{KPS20}, we conclude that for every class of bounded expansion $\Cc$ and positive integer $d$, 
the class $\Cc^d$ has the strong Erd\H{o}s-Hajnal property
Similarly, Conjecture~\ref{conj:lsc}, if true, would imply a weaker variant of Theorem~\ref{thm:main}, where $A=B=V(G)$ and we only require that $\min(|A'|,|B'|)\geq n^{1-\epsi}$.
However, note that in the proof of Theorem~\ref{thm:lsc-EH} we only relied on the existence of a suitable coloring for $p=1$, instead of all positive integers $p$, as postulated by
Conjecture~\ref{conj:lsc}.
While Conjecture~\ref{conj:lsc} remains open in general, its restriction to the case of $p=1$ was very recently proved by Ne\v{s}et\v{r}il et al.~\cite{NOdMPZ19}. 
Therefore, the aforementioned weaker form of Theorem~\ref{thm:main} indeed follows from the results of~\cite{NOdMPZ19} through the reasoning explained above.

\end{document}